\documentclass[9pt,a4paper,reqno]{amsart}
\usepackage{mathtools}
\usepackage{hyperref}
\usepackage{tikz-cd}
\hypersetup{colorlinks,citecolor=blue,filecolor=black,linkcolor=black,urlcolor=blue} 
\usepackage[margin=4cm]{geometry}

\usepackage{xcolor}

\usepackage[backend=bibtex]{biblatex}
\newcommand{\R}{{\mathbb R}}
\newcommand{\CC}{{\mathbb C}}
\newcommand{\HH}{{\mathbb H}}
\newcommand{\W}{\mathcal{W}}

\theoremstyle{plain}
\newtheorem{theorem}{Theorem}[section]
\newtheorem{proposition}[theorem]{Proposition}
\newtheorem{lemma}[theorem]{Lemma}

\theoremstyle{definition}
\newtheorem{definition}[theorem]{Definition}
\newtheorem*{question}{Question}

\theoremstyle{remark}
\newtheorem{remark}[theorem]{Remark}
\newtheorem{example}[theorem]{Example}

\newcommand{\rk}[1]{\mathrm{rank}({#1})}

\title{GKM actions on cohomogeneity one manifolds}

\begin{document}

\author{Oliver Goertsches}

\address{Philipps-Universit\"at Marburg\\
Fachbereich Mathematik und Informatik\\
Hans-Meerwein-Strasse~6\\
35032 Marburg\\
Germany}

\email{goertsch@mathematik.uni-marburg.de}

\author{Eugenia Loiudice} 

\address{Universit\"at Greifswald\\
Institut f\"ur Mathematik und Informatik\\
Walther-Rathenau-Str. 47\\
17489 Greifswald\\
Germany}

\email{eugenia.loiudice@uni-greifswald.de}

\author{Giovanni Russo}

\address{Universit\`a degli Studi dell'Aquila \\
Dipartimento di Ingegneria e Scienze dell'Informazione e Matematica\\
via Vetoio 1 \\
67100 L'Aquila\\
Italy}

\email{giovanni.russo1@univaq.it}

\begin{abstract}
We consider compact manifolds $M$ with a cohomogeneity one action of a compact Lie group $G$ such that the orbit
space $M/G$ is a closed interval.
For $T$ a maximal torus of $G$, we find necessary and sufficient conditions on the group diagram of $M$ such that
the $T$-action on $M$ is of GKM type, and describe its GKM graph. 
The general results are illustrated on explicit examples.
\end{abstract}

\maketitle

\tableofcontents

\section*{Introduction}
\label{introduction}

Consider a compact connected orientable smooth manifold $M$ 
with the action of a compact torus $T$ and vanishing 
cohomology in odd degree (either with real or rational coefficients).
The $T$-manifold $M$ is called a \emph{GKM manifold} if the orbit space $M_1/T$ 
of the union $M_1$ of at most one-dimensional orbits is homeomorphic to a graph. 
This graph, equipped with a labelling encoding the isotropy groups on $M_1$, 
is called \emph{GKM graph}.

The theory of GKM manifolds is named after Goresky--Kottwitz--MacPherson \cite{gkm}.
They observed that the Chang--Skjelbred Lemma \cite[Lemma 2.3]{chang_lemma} yields a 
combinatorical description of the (equivariant) cohomology ring $H_T^*(M)$ for such $T$-spaces 
completely determined by the GKM graph, often called the \emph{GKM description} of the (equivariant) cohomology.
One can regard GKM manifolds as generalisations of toric or quasi-toric manifolds, 
with actions of tori whose dimension is independent of that of $M$.  

A natural question is how additional topological or geometric structure is encoded in the GKM graph of a GKM manifold. 
For example, isometric GKM actions on Riemannian manifolds with positive or non-negative sectional curvature were considered by Goertsches--Wiemeler \cite{oli_wiemeler_2, oli_wiemeler_1}. For special properties of GKM actions on almost complex, symplectic or K\"ahler manifolds see Goertsches--Konstantis--Zoller \cite[Section 2]{gkz}. Most relevant for this paper are results by Guillemin--Holm--Zara \cite{guillemin_holm_zara, guillemin_zara}, who considered the case that the GKM action extends to a transitive action of a compact Lie group. More specifically, one has the following criterion \cite[Theorem 1.1]{guillemin_holm_zara}: 
let $G$ be a compact, semisimple, connected Lie group $G$ acting transitively on a manifold $M$. 
Then $M$ is a GKM manifold for the action of a maximal torus $T\subset G$ if and only if its Euler characteristic is non-zero or, equivalently, when $M = G/K$ for some closed subgroup $T\subset K \subset G$ of maximal rank. Moreover, the authors compute the GKM graph for the $T$-action. Its vertex set is in one-to-one correspondence with the quotient of Weyl groups $\W(G)/\W(K)$, and it is acted on in a natural way by $\W(G)$. Also, it is shown how the existence of $G$-invariant (almost) complex structures on $M$ can be deduced by properties of 
the graph \cite[Theorem 1.2, Theorem 1.3]{guillemin_holm_zara}. 

As the homogeneous case is now understood, we study the following
\begin{question}
Given a compact \emph{cohomogeneity one} $G$-manifold $M$, 
when does a maximal torus $T\subset G$ act on $M$ in a GKM fashion? Further, if the action is of GKM type, what is the structure of the corresponding GKM graph?
\end{question}
In order for the $T$-action to be of GKM type, the orbit space $M/G$ is necessarily a closed interval (Remark \ref{orbit_type},
cf.\ the classification by Mostert \cite{mostert}), so that all but two $G$-orbits are principal, 
and $M$ is determined by a group diagram.

One of the GKM conditions is that the odd cohomology of $M$ vanishes. 
This condition was understood for cohomogeneity one manifolds in \cite[Corollary 1.3, 1.4]{oliver_mare_1}, in terms of the ranks of the occurring isotropy groups. In particular, one or both of the non-principal orbits have to be equal rank homogeneous spaces, 
so they are themselves GKM manifolds, as recalled above \cite{guillemin_holm_zara}.
To understand the remaining GKM conditions, we need to combine this result with information 
on the weights of the slice representation at $T$-fixed points. 
We prove general statements on cohomogeneity one representations (Proposition \ref{prop:cohom1repweylgroup}), 
in order to obtain as our main result (Theorem \ref{thm:gkmconditioncohom1}) an easy-to-check criterion on the group diagram of the cohomogeneity one manifold to 
determine when the action of a maximal torus of $G$ is GKM.
Equivalently, these are conditions for the $G$-action to satisfy the non-abelian GKM conditions as in 
Goertsches--Mare \cite{goertsches_mare}. Notice also that the question above was considered for a variant of GKM theory in odd dimensions in \cite[Section 5.4]{He}.

We give concrete recipes to determine the GKM graph, 
depending on whether one or two of the non-principal orbits are equal rank homogeneous spaces.
We then illustrate our theory by applying it to numerous examples from various sources \cite{ale_pode, hoelscher, podesta_spiro}, 
including cohomogeneity one manifolds with a fixed point, and of dimension up to six.

The paper is organised as follows. In Section \ref{defs} we recall what we need about the theory of cohomogeneity one manifolds and GKM actions.
We also recall some of the statements in Guillemin--Holm--Zara \cite{guillemin_holm_zara} which will be used for our descriptions.
In Section \ref{core} we prove as our main result a characterisation of GKM actions 
in terms of the group diagram of the manifold. Section \ref{gkm_graph_sec} is dedicated to the structure of the GKM graphs of cohomogeneity one actions. Finally, Section \ref{sec:examples}
contains explicit examples of GKM actions on cohomogeneity one manifolds.

\section{Preliminaries}
\label{defs}

We start by recalling the definition of a cohomogeneity one $G$-manifold, the correspondence with group diagrams,
the definition of Weyl group of a cohomogeneity one $G$-action, and the twist of a normal geodesic. Next, we give a 
precise definition of GKM action of a torus on a manifold and discuss some basic general properties. 
Finally we recall some relevant results on the homogeneous case
that we will need in the other sections.

\subsection{Cohomogeneity one manifolds} 
\label{subsec:cohom1}

A \emph{cohomogeneity one} $G$-manifold is a smooth manifold $M$ with the action of a Lie group $G$ 
admitting an orbit of codimension one \cite{ale_ale}.
We take $G$ to be compact and connected, so that $M$ admits a $G$-invariant Riemannian metric. 
We also assume $M$ compact and connected, thus the quotient $M/G$ is a one-dimensional topological Hausdorff space homeomorphic to either a closed interval
$I$ or a circle \cite{mostert}. In the following we always assume $M/G = I$
(cf.\ Remark \ref{orbit_type}). 
We have two non-principal $G$-orbits 
on $M$ corresponding to the end-points of $I$ via the canonical projection
$M \to I$, and the remaining orbits are principal.

Recall that a \emph{group diagram} is a quadruple of compact Lie groups $(G,K^+,K^-, H)$, with $H \subset K^\pm \subset G$, 
such that the quotient spaces $K^\pm/H$ are diffeomorphic to spheres. 
One can associate (not uniquely) a group diagram to any cohomogeneity one manifold in the following way: 
fix an auxiliary $G$-invariant Riemannian metric on $M$, as well as a geodesic $\gamma \colon [0,1]\to M$
running orthogonally to the $G$-orbits, meeting every orbit exactly once, such that $G\cdot \gamma(0)$ and $G\cdot \gamma(1)$ are the two non-regular orbits.
We let $K^+\coloneqq G_{\gamma(0)}$ and $K^-\coloneqq G_{\gamma(1)}$ be the stabilisers of the end-points of $\gamma$, 
and set $H\coloneqq G_{\gamma(t)}$ for $0<t<1$ to be the stabiliser of all other points of $\gamma$ (which is independent of $t$). 
Then $(G,K^+, K^-, H)$ is a group diagram \cite[Theorem 7.1]{ale_ale}.

Conversely, a group diagram $(G,K^+,K^-,H)$ defines a cohomogeneity one $G$-manifold $M$. 
To see this, recall that any transitive Lie group action on a sphere is linear \cite[Theorem 10.1, Chapter II]{bredon}.
This means that there are orthogonal cohomogeneity one $K^\pm$-representations on real vector spaces $V^\pm$ with regular isotropy $H$. 
Then one constructs $M$ by equivariantly gluing the disc bundles $ G\times_{K^\pm} D^\pm$, where $D^\pm\subset V^\pm$ are unit discs. 

Let us move to the notion of \emph{Weyl group of an isometric cohomogeneity one action} of $G$ on a Riemannian manifold $M$ \cite{ale_ale, palais_terng}.  
Take a normal geodesic $\gamma$ on $M$. By definition, the Weyl group of this action is the quotient of the subgroup of $G$ sending $\gamma$ to itself by its pointwise stabiliser $H$. 
Thus the Weyl group can be identified with a subgroup of the quotient $N_G(H)/H$, where $N_G(H)$ is the normaliser
of $H$ in $G$. It can be shown that when $M/G$ is an interval, the Weyl group is a dihedral group
of even order generated by the only two elements of order two inside $N_{K^{\pm}}(H)/H$ \cite[Theorem 5.1]{ale_ale}.
Furthermore, the cardinality of the Weyl group is the number of minimal geodesic segments intersecting the regular part of $M$ \cite{ziller}, so that
the Weyl group is finite if and only if normal geodesics are closed.
The number of times the curve $\gamma$ hits the orbit $G/K^+$ is called the \emph{twist} of $\gamma$ \cite[Definition~5]{ale_ale}.

\subsection{GKM actions}\label{sec:gkmtheory} 

Consider the action of a compact torus $T=S^1\times \cdots \times S^1$ 
on a compact connected orientable manifold $M$. 
In the next sections we will consider $T$ to be a maximal torus in a Lie group $G$ 
acting on $M$ with cohomogeneity one.

\begin{definition}[Goresky--Kottwitz--MacPherson \cite{gkm}]
\label{defn:gkm}
A $T$-action on a compact connected orientable manifold $M$ is said to be of \emph{GKM type} if
\begin{enumerate}
\item the odd cohomology $H^{\mathrm{odd}}(M)$ vanishes,
\item the fixed point set $M^T$ is non-empty and finite,
\item the one-skeleton $M_1 \coloneqq \{p\in M: \dim(T\cdot p) \leq 1\}$ is a finite union of $T$-invariant two-spheres.
\end{enumerate}
In this case we also say that $M$ is a \emph{GKM manifold}.
\end{definition}
Throughout this paper we consider cohomology with either real or rational coefficients. 
For the first condition in Definition \ref{defn:gkm}, it is irrelevant which of these coefficient fields we choose. 
Notice that for $T$-actions with finite fixed point set, the vanishing of the odd cohomology of $M$ is equivalent 
to the equivariant formality of the action, i.e.\ the condition that the equivariant cohomology
$$H^*_T(M) \coloneqq H^*(M\times_T ET)$$
is a free $H^*(BT)$-module, where the module structure is induced by the projection map $M\times_T ET \to BT$.
Here $ET\to BT$ is the classifying bundle of $T$. 

An immediate consequence of Definition \ref{defn:gkm} is that the dimension of a GKM manifold is necessarily even: since the isotropy representation at $p \in M^T$ splits as 
\begin{equation}
\label{iso_rep}
T_pM = T_p(M^T) \oplus \bigoplus_{\alpha} V_{\alpha},
\end{equation}
where the weight spaces $V_{\alpha}$ are two-dimensional irreducible real
$T$-modules, the fixed point set always has even codimension. Now the second condition in the above definition implies $M$ has even dimension.

The third condition in Definition \ref{defn:gkm} can be reformulated in two ways.
Firstly, it is equivalent to the condition that the orbit space $M_1/T$ be homeomorphic to a graph. 
Secondly, it is equivalent to the condition 
\begin{enumerate}
\item[(3')] at each fixed point $p\in M^T$ the weights of 
the isotropy representation on $T_pM$ are pairwise linearly independent. 
\end{enumerate}
Here the weights are elements in ${\mathfrak{t}}^*/\pm 1$, i.e.\ linear forms on the Lie algebra ${\mathfrak{t}}$ 
of the torus $T$ which are well-defined up to sign 
(as we do not assume the existence of an invariant almost complex structure on $M$).

As mentioned above, the orbit space $M_1/T$ has the structure of a graph. 
The fixed points correspond to the vertices of the graph, and every $T$-invariant two-sphere yields an edge, 
connecting the two vertices given by the two fixed points in this sphere. We call this structure the \emph{GKM graph} 
of the $T$-action \cite{gkm, gz}, and we give it the following labelling.
For any $T$-invariant two-sphere $N\subset M$ we have a character $T\to T/H\cong S^1$, where $H$ is the kernel of the $T$-action on $N$. 
Its differential $\lambda \colon \mathfrak{t} \to \R$, viewed as an element of ${\mathfrak{t}}^*/\pm 1$, 
is a weight of the isotropy representations at both $T$-fixed points in $N$ and vanishes on $\mathfrak t \cap \mathfrak h$. 
We then label the edge $N/T$ with the weight $\lambda$.

The equivariant and ordinary cohomology ring, as well as the (equivariant) characteristic classes of a GKM manifold
are determined by its GKM graph. Explicitly, we have an isomorphism (cf.\ \cite[Theorem 7.2]{gkm})
\begin{equation}\label{eq:eqcohomgkm} H^*_T(M) = \left\{\left.(f_p)\in \bigoplus_{p\in M^T} H^*(BT) \, \right| \, \alpha \mid (f_p-f_q)\, \text{for every edge } pq  \text{ with label } \alpha \right\}
\end{equation}
realised by the natural restriction map $H^*_T(M)\to H^*_T(M^T)$. 
Here, we identify $H^*(BT)$ with the ring of polynomials on the Lie algebra ${\mathfrak{t}}$.

\subsection{The homogeneous case}

It was observed by Guillemin--Holm--Zara \cite{guillemin_holm_zara} that a homogeneous space $M=G/K$ of compact Lie groups
with $\rk{G} = \rk{K}$ is a GKM manifold with respect to the action of a maximal torus $T\subset K\subset G$.
Furthermore, they determined the GKM graph explicitly in terms of the root systems of $G$ and $K$. Let us recall their description of the graph. 

For a homogeneous space $M = G/K$ we have the following algebraic characterisation of the Euler characteristic $\chi(M)$ in terms
of the Weyl groups $\W(G)$ and $\W(K)$ of $G$ and $K$ (cf.\ \cite{wang}). 

\begin{proposition}
\label{euler_weyl}
Let $M = G/K$ be a homogeneous space.
\begin{enumerate}
\item If $\rk{G} = \rk{K}$ then $\chi(M) = |\W(G)|/|\W(K)|$.
\item If $\rk{G} \neq \rk{K}$ then $\chi(M) = 0$.
\end{enumerate}
\end{proposition}
The result can be shown in this way: in case of unequal ranks, $M$ admits a vector field without zeros. 
In case of equal ranks, the fixed point set of a maximal torus $T\subset K$ is given by the following

\begin{proposition}[Guillemin--Holm--Zara {\cite[Proposition 2.2]{guillemin_holm_zara}}]
Let $M = G/K$ be a homogeneous space with $\rk{G}=\rk{K}$. Let $T$ be a maximal torus in $K$. Then 
$$M^T = N_G(T)/N_K(T) =\W(G)/\W(K).$$
In particular, $\W(G) = N_G(T)/T$ acts transitively on $M^T$.
\end{proposition}
Then Proposition \ref{euler_weyl} follows by invoking the classical result of Kobayashi \cite{kobayashi} that 
for any $T$-action on a manifold $M$ we have the equality of Euler characteristics 
$$\chi(M) = \chi(M^T).$$ 

Let $\Delta_K\subset \Delta_G$ denote the sets of roots of $K$ and $G$ with respect to the chosen maximal torus 
$T\subset K$, and put $\Delta_{G,K} \coloneqq \Delta_G \setminus \Delta_K$. 
Any root $\alpha$ is a linear form on ${\mathfrak{t}}$, and (fixing a bi-invariant inner product on the Lie algebra ${\mathfrak{g}}$ of $G$) we denote the corresponding reflection with respect to the hyperplane orthogonal to $\alpha$
by $\sigma_\alpha$. Recall also that the Weyl group $\W(G)$ acts on ${\mathfrak{t}}$. The induced action on the dual Lie algebra ${\mathfrak{t}}^*$ is given by 
$$w\cdot \alpha \coloneqq \alpha\circ {\mathrm{Ad}}_w^{-1}.$$ 
Then we have

\begin{theorem}[Guillemin--Holm--Zara {\cite[Theorem 2.4]{guillemin_holm_zara}}]
\label{thm:ghzhomspace}
The $T$-action on $M=G/K$ is of GKM type when $\rk{G} = \rk{K}$. Its GKM graph is given by the following properties:
\begin{enumerate}
\item The set of vertices is $\W(G)/\W(K)$. Vertices are denoted by $[w]$, where $w\in \W(G)$.
\item The edges containing a vertex $[w]$ are in one-to-one correspondence with $\Delta_{G,K}/\pm 1$: for $\alpha\in \Delta_{G,K}/\pm 1$, there is an edge between $[w]$ and $[w\sigma_\alpha]$ labelled  $w\cdot \alpha$.
\end{enumerate}
\end{theorem}

\section{Torus actions on cohomogeneity one manifolds}
\label{core}

From now on, $M$ will always be a compact, connected, orientable cohomogeneity one $G$-manifold with orbit space an interval and group diagram $(G,K^+,K^-,H)$. 
Let $T \subset G$ be a maximal torus. The goal of this section, to be achieved in Theorem \ref{thm:gkmconditioncohom1}, 
is to find necessary and sufficient conditions for the $T$-action on $M$ to be of GKM type in terms of the group diagram of $M$. 
Notice that the GKM condition is independent of the chosen maximal torus of $G$. 

\begin{proposition}
Assume the $T$-action on $M$ is of GKM type. Then the rank of $H$ is smaller than the rank of $G$.
\end{proposition}

\begin{proof}
If the ranks of $G$ and $H$ were the same then every principal orbit would contain a fixed point,
contradicting the second condition in Definition~\ref{defn:gkm}.
\end{proof}

\begin{remark}
\label{orbit_type}
If $M/G$ was a circle, then all $G$-orbits would be principal. Hence any maximal torus of $G$ 
would have either zero or infinitely many fixed points, and $M$ would not be GKM.
\end{remark}

Recall from \cite[Proposition 5.1]{oliver_mare_1} that for a cohomogeneity one 
manifold positivity of the Euler characteristic is equivalent to the vanishing of the odd cohomology (alternatively, one may apply a result of Grove--Halperin \cite{grove_halperin}, see \cite[Remark 5.4]{oliver_mare_1}). 
By \cite[Proposition 1.2.1]{ale_pode} and Proposition \ref{euler_weyl}, when the rank of $H$ is smaller than that of $G$, the Euler characteristic of $M$ 
computes as 
\begin{align*}
\chi(M) & = \chi(G/K^-)+\chi(G/K^+)-\chi(G/H) \\
& = \chi(G/K^-)+\chi(G/K^+),
\end{align*}
so $\chi(M)>0$ if and only if 
one of $K^\pm$ has the same rank as $G$. 

Since $K^{\pm}/H$ are spheres, the differences of the ranks of $K^{\pm}$ and $H$ are at most one \cite[Satz IV]{samelson}. 
Then up to switching the roles of $K^+$ and $K^-$ there are two cases:
\begin{enumerate}
\item $\rk{G} = \rk{K^+} = \rk{K^-} = \rk{H} + 1$,
\item $\rk{G} = \rk{K^+} = \rk{K^-} + 1 = \rk{H}+1$.
\end{enumerate}

In both cases, Conditions (1) and (2) in Definition \ref{defn:gkm} are satisfied. To check when the action is GKM we will investigate Condition (3'). 
Roughly, the strategy to understand it is as follows: 
as each $T$-fixed point $p\in M^T$ is contained in one of the non-principal orbits, 
there is an orthogonal decomposition of the isotropy representation into its tangential and normal part
$$T_pM = T_p(G\cdot p) \oplus \nu_p(G\cdot p),$$
so that the weights can be understood by investigating these two representations separately. 
The tangential representation was already treated in \cite{guillemin_holm_zara}, 
so we need to understand when the weights of cohomogeneity one representations are pairwise linearly independent to each other, 
and when they are pairwise linearly independent to those of the tangential representation.

In the following proposition the group $K$ will be the identity component of either one of those 
subgroups $K^{\pm}$ with $\rk{G} = \rk{K^\pm}$, and $V$ will play the role of the normal space $\nu_p(G\cdot p)$ above. 
Recall that an irreducible representation of a compact Lie group on a finite-dimensional complex vector space is called 
\begin{itemize}
\item {\emph{of real type}} if it is the complexification of a real representation,
\item {\emph{of quaternionic type}} if it is obtained from an ${\mathbb{H}}$-linear representation on a quaternionic vector space by restriction of scalars,
\item {\emph{of complex type}} if it is neither real nor quaternionic. 
\end{itemize}
\begin{proposition}
\label{prop:cohom1repweylgroup}
Let $\pi \colon K\to \mathrm{SO}(V)$ be an orthogonal representation of a compact connected Lie group $K$ on an even-dimensional Euclidean vector space $V$. Assume $K$ acts transitively on the unit sphere in $V$. 
\begin{enumerate}
\item If the representation does not admit an invariant complex structure, then the Weyl group $\W(K)$ 
acts transitively on the weights of $V$ and its complexification $V^\CC$. In this case, the following conditions hold true:
\begin{enumerate}
\item  the weights of the representation $V$ are pairwise linearly independent,
\item for every weight $\lambda$ of $V^\CC$ there exists $w\in \W(K)$ such that $w\lambda = -\lambda$,
\item no weight is a multiple of a root of $K$.
\end{enumerate} \label{1}
\item If the representation admits an invariant complex structure, then we fix one and consider the weights as linear forms on ${\mathfrak{t}}$. 
The Weyl group $\W(K)$ acts transitively on these weights. In this case, the following conditions are equivalent:
\begin{enumerate}
\item the weights are pairwise linearly independent, 
\item $w\lambda\neq -\lambda$ for all $w\in \W(K)$ and all weights $\lambda$, 
\item no weight is a multiple of a root of $K$,
\item the representation is not of quaternionic type,
\item up to non-effectivity, the representation is not isomorphic to the standard representation of ${\mathrm{Sp}}(n)$ on ${\mathbb{H}}^n$. 
\end{enumerate}\label{2}
\end{enumerate}
\end{proposition}
\begin{proof}
Our first goal is to show the transitivity statements on the Weyl group actions in \eqref{1} and \eqref{2}. 
A large part of this bit of the proof is contained in \cite[Lemma 29]{gorodski}. To this end, consider two weights $\lambda,\mu\in {\mathfrak{t}}^*/\pm 1$ of $V$. Let $v,w$ be unit length weight vectors for $\lambda$ and $\mu$ respectively. As the action of $K$ is transitive on the unit sphere in $V$, there is $a\in K$ such that $w=av$. The isotropy algebras $\mathfrak{k}_v$ and $\mathfrak{k}_w$ are hence conjugate via $a$, namely $\mathfrak{k}_w={\mathrm{Ad}}_a(\mathfrak{k}_v)$. Let $\mathfrak{t}_v=\ker \lambda$, $\mathfrak{t}_w=\ker \mu$ be the isotropy algebras of the restricted action of $T$ on $V$. Observe that by the theorem of maximal tori, ${\mathrm{Ad}}_a(\mathfrak{t}_v),\mathfrak{t}_w\subset \mathfrak{k}_w={\mathrm{Ad}}_a(\mathfrak{k}_v)$ are conjugate via an element $a'\in K_w$, namely $\mathfrak{t}_w={\mathrm{Ad}}_{a'a}(\mathfrak{t}_v)$. Thus, we have that $\ker \mu$ is contained in ${\mathrm{Ad}}_{a'a}(\mathfrak{t})$ and  $\mathfrak{t}$, which are both Lie algebras of maximal tori in $\mathfrak{k}$.  Notice that any two maximal tori $T_1$, $T_2$ of a compact Lie group are conjugate to each other via a transformation that fixes their intersection pointwise: in fact, the centralizer $C(T_1\cap T_2)$ of the intersection of $T_1$ and $T_2$ contains both tori, so $T_1$ and $T_2$ are conjugate via an element in $C(T_1\cap T_2)$.
Therefore, we take an element $a''\in K$ such that ${\mathrm{Ad}}_{a''}$ fixes $\ker \mu$ and satisfies ${\mathrm{Ad}}_{a''}({\mathrm{Ad}}_{a'a}({\mathfrak{t}})) = {\mathfrak{t}}$.  Then, $a''a'a$ defines an element of the Weyl group $\W(K)$ mapping $\ker \lambda$ to $\ker \mu$. 
In particular, all weights are conjugate to a multiple of $\lambda$ via the $\W(K)$-action. 

Now, we denote by $U$ the irreducible complex representation which is either $V$ with a fixed invariant complex structure, or $V^\CC$ in case $V$ does not admit any.
We will show that $\W(K)$ acts transitively on the weights of $U$ (which are linear forms on ${\mathfrak{t}}^*$). This will imply all the desired transitivity statements, as the weights of $V^\CC$ are exactly all linear forms $\lambda$ such that $\pm \lambda$ is a weight of $V$. 

Let $\lambda$ be a highest weight of the representation $U$ (see e.g.\ \cite[Ch.\ IX, \S 7.2]{bourbaki} for a reference that treats this theory for representations of compact groups). We choose a notion of positivity such that $\lambda$ is contained in the closure of the positive Weyl chamber (recall that the positive, or fundamental, Weyl chamber is defined by $C=\{\beta\in {\mathfrak{t}}^*\mid \langle \beta,\alpha\rangle >0 \textrm{ for all simple roots }\alpha\}$). If all (simple) roots are perpendicular to $\lambda$, then $\lambda$ is the only weight of $U$, and hence $U$ is effectively a circle representation on $\CC$, for which the transitivity statement is trivial. So we may assume that there is at least one simple root which is not perpendicular to $\lambda$. Let us consider first the case that one simple root is parallel to $\lambda$, and all others perpendicular. In this case, the parallel root belongs to an isolated node in the Dynkin diagram of $K$, i.e.\ up to covering $K$ splits off an $\mathrm{SU}(2)$ factor, and the restriction of the representation to this $\mathrm{SU}(2)$ is already transitive on the unit sphere in $U$. As $\mathrm{SU}(2)$ is three-dimensional, this implies that the complex dimension of $U$ is two. The only weights are hence $\pm \lambda$, and the Weyl group of $\mathrm{SU}(2)$ acts transitively on these weights.

In the remaining case there exists a simple root $\alpha$ of $K$ that is not parallel to $\lambda$ and satisfies $\langle \alpha,\lambda\rangle >0$. The weights are invariant under reflection along roots, and 
\[
\lambda,\, \lambda - \alpha,\, \ldots, \, \lambda- 2\frac{\langle \lambda,\alpha\rangle}{\lVert\alpha \rVert^2}\alpha = s_\alpha \lambda
\]
are weights of $U$---this follows from the classification of irreducible ${\mathrm{SU}}(2)$-representations. By definition of the positive Weyl chamber $C$ the hyperplane defined by $\alpha$ touches $C$. Since $\lambda$ is contained in the closure $\overline{C}$ of $C$, and $s_\alpha\lambda$ in $\overline{s_\alpha(C)}$, this string of weights is contained in the convex cone $\overline{C}\cup \overline{s_\alpha(C)}$. As the Weyl group acts simply transitively on the Weyl chambers, and we have shown already that it acts transitively on the weights up to multiple, this shows that the string contains at most three elements. However, it can only contain three elements, if they are of the form $\lambda,\, \lambda-\alpha,\, \lambda-2\alpha$, and lie on the boundaries of the Weyl chambers. But this would yield a contradiction as follows: denoting by $w$ a Weyl group element sending $s_\alpha\lambda=\lambda-2\alpha$ to a multiple of $\lambda-\alpha$, then $s_\alpha\circ w$ fixes $s_\alpha(C)$ but not pointwise. Thus, the string contains only two elements, i.e.\ $s_\alpha \lambda = \lambda-\alpha$. Now, if there existed a weight of the form $c\lambda$, with $0< c< 1$, then $s_\alpha (c\lambda) = c(\lambda - \alpha) = c\lambda-c\alpha$, which is a contradiction, as $c$ is not an integer. Then the weights are all of maximal length $|\lambda|$. Since $\lambda$ is the unique weight of maximal length in $\overline{C}$ and $\W$ acts transitively on the Weyl chambers, it acts transitively on the weights of maximal length and hence on all weights.

Let us consider case \eqref{1}. Since any two weights of a representation related to each other by a Weyl group element have the same multiplicity (a representative of the Weyl group element in the normaliser of the torus intertwines the corresponding weight spaces),
we have that the weight spaces of the complexified representation $V^\CC$ are all complex one-dimensional. It follows that the weight space of any weight $\pm \lambda$ is real two-dimensional, i.e.\ \eqref{1} (a) holds true. 

We observe that for any weight $\lambda$ of $V^\CC$ also $-\lambda$ is a weight. As the Weyl group $\W(K)$ acts transitively on the weights of $V^\CC$, statement \eqref{1} (b) is immediate.

To show \eqref{1} (c), i.e.\ that no weight is a multiple of a root of $K$, we assume the contrary. As we showed above that the Weyl group acts transitively on the weights, it follows that every weight is a multiple of a root (but not necessarily all roots multiples of weights). As all weights have the same length, for every weight $\lambda$ there is a root $\alpha$ such that $\lambda= \alpha/2$ ($\lambda$ is a positive multiple of a root $\alpha$, then $\lambda-\alpha$ is also a weight, which has to be $-\lambda$). We fix a weight $\lambda$, together with the root $\alpha = 2\lambda$, and choose a notion of positivity of roots such that $\alpha$ is simple. Let $K'\subset K$ be the simple factor of $K$ whose root system contains $\alpha$. We consider the restriction of the representation to $K'$. As $\lambda$ is an integral weight, we have $\langle \alpha,\beta\rangle/\langle \beta,\beta\rangle = 2\langle \lambda,\beta\rangle / \langle \beta,\beta\rangle \in {\mathbb{Z}}$ for all roots $\beta$. Varying $\beta$ over the other simple roots of $K'$, these numbers are exactly $1/2$ times the entries in the $\alpha$-column in the Cartan matrix of $K'$. The only simple Lie group $K'$ whose Cartan matrix has a column with only even entries is ${\mathrm{Sp}}(n)$, i.e.\ $C_n$, for $\alpha$ the long simple root. In this case, $\lambda = \alpha/2$ is an integral weight and the corresponding representation is the standard representation of ${\mathrm{Sp(n)}}$ on ${\mathbb{H}}^n$, which is quaternionic, hence not the complexification of a real representation, which is a contradiction. Hence it is impossible that a weight is a multiple of a root, i.e.\ \eqref{1} (c) holds true.

Consider now the case that the representation admits a complex structure. We show the equivalences in \eqref{2}. If the weights are pairwise linearly independent, then for sure $w\lambda\neq -\lambda$ for all $w$ and all $\lambda$. The converse direction is true because the Weyl group acts transitively on the weights and the weight spaces are one-dimensional. If a weight $\lambda$ was a multiple of a root, then reflection along the root would send $\lambda$ to $-\lambda$. 

Finally, if $w\lambda=-\lambda$ for some weight $\lambda$ and $w\in W(K)$ (and hence for all $\lambda$ there is such $w$, by the transitivity of the Weyl group action), then the representation is quaternionic \cite[Ch.\ IX, \S 7.2, Proposition 1]{bourbaki} (it cannot be of complex type, because these representations do not admit symmetric weights with respect to $0$, and it cannot be of real type, because the complexification of an orthogonal representation has at least cohomogeneity two). 
Then, considering the list of effective cohomogeneity one representations \cite{borel2, borel1, montgomery}, cf.\ \cite{bogdan},
\begin{center}
\begin{table}[htb]
\begin{tabular}{cccccc}
 $K$ & $H$ & $K/H$& \multicolumn{2}{c}{representation}\\[0.05cm]
\hline \\[-0.25cm]
${\mathrm{SO}}(2n+1)$ & ${\mathrm{SO}}(2n)$ & $S^{2n}$ & ${\mathbb{R}}^{2n+1}$ & odd-dimensional\\[0.05cm]
${\mathrm{SO}}(2n)$ & ${\mathrm{SO}}(2n-1)$ & $S^{2n-1}$ & ${\mathbb{R}}^{2n}$ & real\\[0.05cm]
${\mathrm{SU}}(n)$ & ${\mathrm{SU}}(n-1)$ & $S^{2n-1}$ & ${\mathbb{C}}^n$ & complex type\\[0.05cm]
${\mathrm{U}}(n)$ & ${\mathrm{U}}(n-1)$ & $S^{2n-1}$ & ${\mathbb{C}}^n$ & complex type\\[0.05cm]
${\mathrm{Sp}}(n)$ & ${\mathrm{Sp}}(n-1)$ & $S^{4n-1}$ & ${\mathbb{H}}^n$ & quaternionic type\\[0.05cm]
${\mathrm{Sp}}(n)\cdot {\mathrm{Sp}}(1)$ & ${\mathrm{Sp}}(n-1)\cdot {\mathrm{Sp}}(1)$ & $S^{4n-1}$ & ${\mathbb{H}}^n$ & real\\[0.05cm]
${\mathrm{Sp}}(n)\cdot {\mathrm{U}}(1)$ & ${\mathrm{Sp}}(n-1)\cdot {\mathrm{U}}(1)$ & $S^{4n-1}$ & ${\mathbb{H}}^n$ & complex type\\[0.05cm]
${\mathrm{G}}_2$ & ${\mathrm{SU}}(3)$ & $S^6$ & ${\mathrm{Im}}\, {\mathbb{O}}$ & odd-dimensional\\[0.05cm]
${\mathrm{Spin}}(7)$ & ${\mathrm{G}}_2$ & $S^7$ & $\Delta_7 \cong {\mathbb{R}}^8$ & real\\[0.05cm]
${\mathrm{Spin}}(9)$ & ${\mathrm{Spin}}(7)$ & $S^{15}$ & \ $\Delta_9\cong {\mathbb{R}}^{16}$ & real\\[0.2cm]
\end{tabular}
\caption{Transitive actions on spheres / cohomogeneity one representations of compact connected Lie groups. Here with the notation of type ${\mathrm{Sp}}(n)\cdot {\mathrm{Sp}}(1)$ one means the quotient group $({\mathrm{Sp}}(n)\times {\mathrm{Sp}}(1))/\mathbb{Z}_2$, where $\mathbb{Z}_2$ is generated by $(-{\mathrm{id}}, -{\mathrm{id}})$.}
\end{table}
\end{center}
it has to be (up to the non-effectivity kernel) the standard representation of $\mathrm{Sp}(n)$ on $\mathbb{H}^n$ 
(to see that the spin representations of $\mathrm{Spin}(7)$ and $\mathrm{Spin}(9)$ are real see for instance \cite[Section 1.7]{friedrich}), 
in which case the weights are multiple of roots (see also Remark \ref{ex:weightsquat} below).
\end{proof}

\begin{remark}
\label{ex:weightsquat} 
A quaternionic representation may be viewed as a complex representation $V$ together with an equivariant endomorphism $j$ such that $j^2 = -\mathrm{id}$ and $j(iv) = -ij(v)$, $v \in V$.
If $V_{\alpha}$ is a weight space with weight $\alpha$, then $j(V_{\alpha})$ has weight $-\alpha$. 
Hence the weights of a quaternionic representation are never pairwise linearly independent.
For example, this is the case for the standard $\mathrm{Sp}(n)$-representation 
on $\HH^n$. For $n=2$, its weights are depicted in Figure~\ref{fig:sp2h2}.
As predicted by Proposition \ref{prop:cohom1repweylgroup}, the weights are multiples 
of roots (and every weight is sent to its negative by some Weyl group element). 
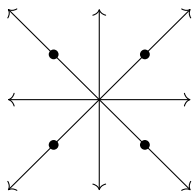
\begin{figure}[htb]
\begin{center}
\begin{tikzpicture}[scale=1.2]
\draw[->] (0,0) -- (1,-1);
\draw[->] (0,0) -- (0,-1);
\draw[->] (0,0) -- (1,1);
\draw[->] (0,0) -- (-1,-1);
\draw[->] (0,0) -- (-1,1);
\draw[->] (0,0) -- (0,1);
\draw[->] (0,0) -- (1,0);
\draw[->] (0,0) -- (-1,0);

\node at (1/2,1/2)[circle,fill,inner sep=1.3pt]{};
\node at (-1/2,1/2)[circle,fill,inner sep=1.3pt]{};
\node at (1/2,-1/2)[circle,fill,inner sep=1.3pt]{};
\node at (-1/2,-1/2)[circle,fill,inner sep=1.3pt]{};
\end{tikzpicture}

\end{center}

\caption{Roots of ${\mathrm{Sp}}(2)$ and weights of its standard representation on $\HH^2$}
\label{fig:sp2h2}
\end{figure}

Notice that there are two other representations on $\mathbb H^n$ in the above list, of $\mathrm{Sp}(n)\cdot \mathrm{Sp}(1)$ and $\mathrm{Sp}(n) \cdot \mathrm{U}(1)$, that are not quaternionic. 
If one of these representations was quaternionic, then the weights with respect to the restricted complex structure 
would come in pairs of the form $\alpha$ and $-\alpha$, hence $\pm \alpha$ would have multiplicity greater
than one as weights of the real representation. However, 
these representations have $2n$ pairwise linearly independent weights: denoting the dual basis of the standard basis of the diagonal maximal torus $T^n\subset {\mathrm{Sp}}(n)$ by $e_i$ and of the diagonal maximal torus $T^1= {\mathrm{U}}(1)\subset {\mathrm{Sp}}(1)$ by $f$, the weights are given by $e_1-f,\ldots,e_n-f,e_1+f,\ldots,e_n+f$.
\end{remark}

We can now characterise cohomogeneity one GKM manifolds in terms of properties of their group diagram.
\begin{theorem}
\label{thm:gkmconditioncohom1} 
Consider an orientable cohomogeneity one $G$-manifold, 
represented by the group diagram $(G,K^+,K^-,H)$. 
Let $T\subset G$ be a maximal torus containing a maximal torus of~$H$, 
and $\Delta_G\subset {\mathfrak{t}}^*$ be the root system of $G$ with respect to $T$. 
Then the $T$-action is of GKM type if and only if 
the following two conditions hold:
\begin{enumerate}
\item[(i)] at least one of the singular isotropy groups $K^{\pm}$ has the same rank as $G$, and $\rk H = \rk G - 1$,
\item[(ii)] no element of $\Delta_G$ vanishes on ${\mathfrak{t}}\cap {\mathfrak{h}}$. 
\end{enumerate}
Furthermore, if the action is of GKM type, then there is a unique maximal torus $T\subset G$ 
containing a given maximal torus of $H$. 
Such $T$ is contained in any of the two subgroups $K^\pm$ that are of maximal rank. 
\end{theorem}

\begin{proof}
We already argued at the beginning of this section that (i) is equivalent to (1) and (2) in the definition of a GKM action, i.e.\
 to the vanishing of the odd cohomology and finiteness of the fixed point set. 
We need to verify that in presence of (i), Condition (3') in Section \ref{sec:gkmtheory} is equivalent to Condition (ii). 
We may assume that $\rk G=\rk{K^+}$ and that the maximal torus $T$ is contained in $K^+$. 
Below we will show that in the GKM case any maximal torus in $H$ uniquely extends to a maximal torus in $G$, 
so it suffices to prove the equivalence for this torus $T$.

We set $p\coloneqq eK^+\in G/K^+$. We first notice that pairwise linear independence of the weights of the isotropy representation at 
all $T$-fixed points in $G/K^+$ is equivalent to this condition at~$p$, by equivariance with respect to the normaliser $N_G(T)$. The isotropy $K^+$-representation decomposes as
$$T_pM = T_p(G/K^+)\oplus \nu_p (G/K^+).$$ 
The weights of $T_p(G/K^+)$ are given by $(\Delta_G\setminus \Delta_{K^+})/\pm$ and are hence pairwise linearly independent. 
Notice that the normal representation on $\nu_p(G/K^+)$ has as a weight a linear form modulo $\pm 1$
with kernel ${\mathfrak{t}} \cap {\mathfrak{h}}$, as $T\cap H$ occurs as a codimension one isotropy. 
Thus, the question whether a linear form is linearly independent from this weight is equivalent to ask whether this linear 
form does not vanish on ${\mathfrak{t}} \cap {\mathfrak{h}}$. As the Weyl group $\W(K^{+}_{0})$ of the identity 
component $K^+_0$ of $K^+$ acts transitively on the weights of $\nu_p(G/K^+)$ by Proposition \ref{prop:cohom1repweylgroup}, 
and $\W(K^{+}_{0})$ leaves invariant $\Delta_G\setminus \Delta_{K^+}$, the condition that the weights of $\nu_p(G/K^+)$ 
are pairwise linearly independent to the weights of $T_p(G/K^+)$ is equivalent to the condition that no element 
in $\Delta_G\setminus \Delta_{K^+}$ vanishes on ${\mathfrak{t}}\cap {\mathfrak{h}}$. Finally, by Proposition \ref{prop:cohom1repweylgroup}, 
in case the normal $K^+$-representation does not admit any invariant complex structure, 
the condition that the weights of $\nu_p(G/K^+)$ are pairwise linearly independent is automatically satisfied, 
and the weights are not multiples of elements in $\Delta_{K^+}$. In case it admits an invariant complex structure, the condition that the weights of $\nu_p(G/K^+)$ 
are pairwise linearly independent is equivalent to no $K^+$-root vanishing on ${\mathfrak{t}}\cap {\mathfrak{h}}$, by Proposition \ref{prop:cohom1repweylgroup}.

The additional uniqueness statement follows because no root of $G$ vanishes on ${\mathfrak{t}}\cap {\mathfrak{h}}$, hence ${\mathfrak{t}}\cap {\mathfrak{h}}$ contains regular elements.
\end{proof}

We now show by an example that the second condition in the above theorem does not follow from the first.

\begin{example}
Consider the standard $S^3$-action on $S^4$ realising it as a cohomogeneity one manifold with group diagram 
$(S^3,S^3,S^3,\{e\})$. This example satisfies the first, but not the second condition in Theorem \ref{thm:gkmconditioncohom1},
as $\mathfrak t \cap \mathfrak h$ is trivial. Indeed, the one-skeleton of $S^4$ for the action of a maximal torus in $S^3$
is all of $S^4$, contradicting Definition \ref{defn:gkm}.
\end{example}

\begin{remark}
\label{rem:nonabeliangkm}
We briefly mention the relation of our results to non-abelian GKM theory, as developed in \cite{goertsches_mare}. 
There, an action of a compact connected Lie group $G$ on a compact connected manifold $M$ is said to satisfy the non-abelian GKM conditions if 
\begin{enumerate}
\item the $G$-action is equivariantly formal,
\item the union $M_{\max}$ of those orbits whose isotropy group has the same rank as $G$ consists of only finitely many orbits,
\item for every $p\in M_{\max}$ the weights of the isotropy representation of $G_p$ on $T_pM$ are pairwise linearly independent.
\end{enumerate} 
It was shown in \cite[Lemma 4.2]{goertsches_mare} that the non-abelian GKM conditions for the action of a compact connected 
Lie group $G$ on a compact connected orientable manifold are equivalent to the ordinary GKM conditions of a maximal torus 
$T\subset G$. This shows that for a cohomogeneity one $G$-manifold $M$, represented by a group diagram $(G,K^+,K^-, H)$, 
the $G$-action on $M$ satisfies the non-abelian GKM conditions if and only if the two conditions in Theorem \ref{thm:gkmconditioncohom1} hold. 
\end{remark}

\section{The GKM graph}
\label{gkm_graph_sec}

 As before, we fix a group diagram $(G,K^+, K^-, H)$ with associated cohomogeneity one $G$-manifold $M$. 
 We choose an auxiliary $G$-invariant Riemannian metric on $M$, as well as a normal geodesic $\gamma \colon \R\to M$ 
 such that the restriction of $\gamma$ to $[0,1]$ meets every $G$-orbit exactly once, with $G_{\gamma(0)} = K^+$, $G_{\gamma(1)} = K^-$, 
 and $G_{\gamma(t)} = H$ for $0<t<1$. The choice of base points $\gamma(0)$ and $\gamma(1)$ identifies 
 the non-principal orbits with the coset spaces $G/K^\pm$. We consider a maximal torus $T\subset G$ containing 
 a maximal torus of $H$ and assume that its action on $M$ is of GKM type (recall that by Theorem \ref{thm:gkmconditioncohom1}, 
 in this situation $T$ is the only maximal torus containing a given maximal torus of $H$). 
 
 In this section we describe the GKM graph of the $T$-action in the two cases distinguished at the beginning of Section \ref{core}.
 The graph contains as subgraphs the GKM graphs of 
 those singular orbits $G/K^\pm$ for which $\rk{K^\pm} = \rk{G}$, which were described in \cite{guillemin_holm_zara}. 
\begin{remark}
It would be interesting to investigate if the existence of a $G$-invariant (almost) complex structure on the cohomogeneity one manifold $M$ is encoded in the GKM graph, as is the case in the homogeneous setting \cite[Theorem 3.2]{guillemin_holm_zara}.
\end{remark}

\subsection{First case} 
\label{subsec:firstcase}

We assume $\rk{K^{\pm}} = \rk{G}$. Because, as recalled above, $T$ is the unique maximal torus containing a maximal torus of $H$, we have $T\subset K^\pm$. Hence, both non-regular orbits $G/K^\pm$ are GKM manifolds and $\gamma(0)$ and $\gamma(1)$ are $T$-fixed points.
We may then apply Theorem \ref{thm:ghzhomspace} to describe the GKM graphs of the two non-regular orbits. 
It remains to understand the $T$-invariant two-spheres not contained in any non-regular orbit. 

As $T$ contains a maximal torus of $H$, the geodesic $\gamma$ lies inside a $T$-invariant two-sphere 
whose two fixed points are precisely the intersection points of $\gamma$ with the two singular orbits. Note that this implies directly that in this case, the Weyl group of the cohomogeneity one action is ${\mathbb{Z}}_2$. Equivalently, the twist of any normal geodesic 
is one (cf.\ Section~\ref{subsec:cohom1}).
This $T$-invariant two-sphere yields an edge connecting the vertices $eK^+$ and $eK^-$ 
whose label is (up to sign) the differential $\lambda$ of the character $T\to T/(T\cap H) \cong S^1$,
i.e.\ $\lambda \in \mathfrak t^*$ with kernel ${\mathfrak{t}}\cap {\mathfrak{h}}$. 

By Proposition \ref{prop:cohom1repweylgroup}, the Weyl group $\W(K^+)$ acts transitively on the 
weights of the slice representation on $\nu_{\gamma(0)}(G/K^+)$. 
Geometrically, the normaliser $N_{K^+}(T)$ acts transitively 
on the set of $T$-invariant two-spheres emerging from $\gamma(0)$ in direction normal to $G/K^+$. 
Let $\W'\subset \W(K^+)$ be 
the stabiliser of the weight $\lambda$.  
Notice that the representatives of $\W'$ leave invariant the two-sphere containing $\gamma$, so 
that $\W'\subset \W(K^-)$. 
Now, recall that $\W(G)$ acts transitively on the $T$-fixed points in $G/K^+$. Hence combining its action
with the $\W(K^+)$-action on the weights of the normal representation at each fixed point, we find all edges departing from the
singular orbit $G/K^+$ and their end-points represented by $T$-fixed points in $G/K^-$. 

\begin{proposition}
\label{prop:normal_weights_1}
The edges in the GKM graph of the $T$-action not contained in any singular orbit are given as follows: for any $w\in \W(G)/\W'$ there is
an edge between the vertices $[w]\in \W(G)/\W(K^+)$ and $[w]\in \W(G)/\W(K^-)$, with label $w\cdot \lambda$.
\end{proposition}

\begin{remark}
The number of fixed points $s$ on the singular orbit $G/K^+$ and $s'$ on $G/K^-$ 
are related to the dimension of the singular orbits by
$$ s(n-k)=s'(n-k'),$$
where $2n=\dim M$, $2k=\dim G/K^+$ and $2k'=\dim G/K^-$, as one sees by counting normal edges in the GKM graph.
\end{remark}

\subsection{Second case}
\label{subsec:secondcase}

Assume $\rk{K^+} = \rk{G}$ and $\rk{K^-} = \rk{H}$. In this case we have $T\subset K^+$, and $\gamma(0)$ is a $T$-fixed point. We may apply Theorem \ref{thm:ghzhomspace} 
to describe the GKM graph of the GKM manifold $G/K^+$.

The geodesic $\gamma$ is contained in a $T$-invariant two-sphere, 
corresponding to an edge in the GKM graph with label $\lambda$ given (up to sign) by the differential of the character $T\to T/H\cong S^1$. 
Unlike in the first case, this edge does not connect $[e] \in \W(G)/\W(K^+)$ with a vertex belonging to the other singular orbit, 
as there are no $T$-fixed points in $G/K^-$. Let us call $T'\subset H\cap T$ the maximal torus in $H$ contained in $T$.

\begin{lemma}
\label{lem:weylquotcase2}
We have $\W(K^-)/\W(H) \cong {\mathbb{Z}}_2$. There is an element $g\in N_{K^-}(T')\cap N_{K^-}(H)$ reflecting the geodesic $\gamma$ in $\gamma(1)$.
\end{lemma}
\begin{proof}
By Proposition \ref{euler_weyl}, the order of $\W(K^-)/\W(H)$ equals the Euler characteristic of the even-dimensional sphere $K^-/H$, which is two. 
This shows the first statement. As $K^-$ acts transitively on the unit sphere in the normal space at $\gamma(1)$, we find an element in $K^-$ sending $\gamma'(1)$ to $-\gamma'(1)$. 
This element normalises $H$, but is not contained in $H$. By the theorem of maximal tori, we can multiply it by some element of $H$ to obtain an element $g$ which additionally normalises $T'$.
\end{proof}
The above statement should be compared with \cite[Lemma 5.2]{ale_ale}.
Let $g$ be an element from Lemma \ref{lem:weylquotcase2}. 
The geodesic $\gamma$ starts from $\gamma(0) = eK^+\in G/K^+$, moves to $\gamma(1)=eK^-$ in $G/K^-$, and goes back to $g\gamma(0)\in G/K^+$, which is a $T$-fixed point by the following
\begin{proposition} 
The maximal torus $T\subset G$ is also contained in $gK^+g^{-1}$.
\end{proposition}
\begin{proof}
By Theorem \ref{thm:gkmconditioncohom1}, the GKM condition implies that $T \subset K^+$. Moreover, the cohomogeneity 
one manifold $M$ is also determined by the group diagram $(G,gK^+g^{-1},K^-,H)$ (consider the reflected geodesic $\gamma$), 
so applying the theorem again, $T$ is also contained in $gK^+g^{-1}$.
\end{proof}
As the order of the cohomogeneity one Weyl group is equal to the number of minimal geodesic segments intersecting the regular part, 
the cohomogeneity one Weyl group is the dihedral group with four elements, and the twist of $\gamma$ is two.

As we know from Theorem \ref{thm:gkmconditioncohom1}, $T$ is the unique maximal torus containing $T'$, 
so the normaliser $N_{K^-}(T')$ is naturally a subgroup of $N_G(T)$, and we obtain a diagram of inclusions
\[
\begin{tikzcd}
\W(H) \arrow[hook]{r} \arrow[hook]{d} & \W(K^+) \arrow[hook]{d} \\
\W(K^-) \arrow[hook]{r}& \W(G) 
\end{tikzcd}
\]
hence a natural map ${\mathbb{Z}}_2\cong \W(K^-)/\W(H) \to \W(G)/\W(K^+)$. 
The image of the non-trivial element under this map, i.e.\ the coset of $g$, is the second fixed point of 
our distinguished normal sphere. Again, as in the first case, the other normal edges are given by letting $\W(G)$ act.

Thus, using the same notation as in the section above we have
\begin{proposition}
The normal edges in the GKM graph of the $T$-action are given as follows: for any $w\in \W(G)/\langle \W',g\rangle$ there is
an edge between the vertices $[w]\in \W(G)/\W(K^+)$ and $[wg]\in \W(G)/\W(K^+)$, with label $w\cdot \lambda$.
\end{proposition}
\begin{remark}
The description of the GKM graph of the $T$-action determines the equivariant cohomology of the cohomogeneity one $G$-action as an $H^*(BT)$-algebra  via \eqref{eq:eqcohomgkm}. Notice that the equivariant cohomology of an arbitrary cohomogeneity one action was determined in \cite{CGHM}.
\end{remark}

\begin{remark}\label{rem:non-abeliangkm2}
Recall from Remark \ref{rem:nonabeliangkm} that the cohomogeneity one manifolds we consider here 
satisfy the non-abelian GKM conditions \cite{goertsches_mare}. See \cite[Remark 6.6]{goertsches_mare} 
for a description of the non-abelian GKM graph of a cohomogeneity one GKM action and 
the consequences for the equivariant cohomology $H^*_G(M)$. 
\end{remark}

\section{Examples}
\label{sec:examples}

In this section we present many examples of cohomogeneity one $G$-manifolds and use Theorem \ref{thm:gkmconditioncohom1} to 
establish when a maximal torus of $G$ acts in a GKM fashion. We also determine the GKM graph of some of these examples.

Let us first give a general description of the pictures we are going to show. 
If the manifold has group diagram $(G,K^+,K^-,H)$ and $T$ is a maximal torus in $G$ containing a maximal torus of $H$, we draw $T$-fixed
points in $G/K^+$ on the left and $T$-fixed points in $G/K^-$ on the right. The continuous
lines give the edges of the homogeneous graphs for the singular orbits, whereas dotted lines denote
the edges corresponding to normal two-spheres. Each vertex is labelled by an element of $\W(G)/\W(K^\pm)$,
 and edges are labelled by weights of the action, according to the above results. 
 
We first consider the special case of a cohomogeneity one $G$-action on a simply-connected compact 
even-dimensional manifold $M$, such that $G$ 
acts 
effectively and with at least one fixed point. Such actions were classified up to 
$G$-equivariant diffeomorphism in \cite[Section 2]{ale_pode} as well as 
\cite[Section 1.21]{hoelscher2}. If the action has two fixed points, then it is given by a 
group
diagram
$$(G,G,G,H),$$
such that $G/H$ is an odd-dimensional sphere. In this case, there is a
representation $G\to \mathrm{SO}(V)$ of cohomogeneity one with principal 
isotropy group $H$, and $M$ is the unit sphere in $V\oplus {\mathbb{R}}$, 
with 
the action $g\cdot (v,t) = (gv,t)$. It follows from Proposition 
\ref{prop:cohom1repweylgroup} that the action of a maximal torus $T\subset G$ 
is 
GKM if and only if the $G$-representation on $V$ is, up to effectivity, not the standard representation of ${\mathrm{Sp}}(n)$ on $\mathbb{H}^n$.
The GKM graph then has exactly two vertices, corresponding to the two $G$-fixed points. For every weight $\lambda$ of the representation $V$ there is one dotted
edge connecting the two vertices with label $\lambda$.

\begin{example}\label{ex:6-sphere}
Consider the six-sphere $S^6$. As ${\mathrm{SU}}(3)/{\mathrm{SU}}(2)=S^5$, the group diagram 
\[
(G,K^+,K^-,H) = (\mathrm{SU}(3),\mathrm{SU}(3), \mathrm{SU}(3), \mathrm{SU}(2))
\]
describes a cohomogeneity one action of ${\mathrm{SU}}(3)$ on $S^6$, where we consider ${\mathrm{SU}}(2)$ embedded as the upper left block in ${\mathrm{SU}}(3)$. The two-dimensional diagonal maximal torus $T\subset {\mathrm{SU}}(3)$ acts in a GKM fashion by the argument above, and its GKM graph is given in the first configuration in Figure \ref{fig:graphs6}. Explicitly, none of the roots $e_1-e_2,e_2-e_3 = e_1+2e_2,e_1-e_3 = 2e_1+e_2$ of ${\mathrm{SU}}(3)$ vanishes on the Lie algebra of the diagonal maximal torus in ${\mathrm{SU}}(2)$. Here, as usual, the $e_i$'s denote the elements of the dual basis of the standard basis of the Lie algebra of the diagonal torus in ${\mathrm{U}}(3)$, as well as their restrictions to ${\mathfrak{t}}$. The weight of the isotropy representation vanishing on this Lie algebra is $\pm e_3 = \pm(e_1+e_2)$, and the other two weights $\pm e_1$ and $\pm e_2$ are obtained by applying the Weyl group $\W({\mathrm{SU}}(3))$, which is the permutation group $S_3$. 

The sphere $S^6$ also admits a cohomogeneity one action without fixed points \cite{podesta_spiro},
with group diagram 
$$(\mathrm{SU}(2)^2,\mathrm{SU}(2)\times S^1,\Delta\mathrm{SU}(2),\Delta S^1),$$
where $\Delta$ denotes a diagonal embedding in $\mathrm{SU}(2)^2$. This action appears as Example $2_6H$ in Table \ref{Hoelscher2} below.
Observe that both conditions of Theorem~\ref{thm:gkmconditioncohom1} are satisfied, and hence also in this case we have a GKM action of a two-dimensional torus $T^2$ on $S^6$. We describe now its GKM graph. 
The $T^2$-fixed points are both contained in the singular orbit $G/K^+$ (hence we are in Case~\ref{subsec:secondcase}). Observe that
the quotient $\W(G)/\W(K^+)$
is isomorphic to $(\mathbb Z_2 \times \mathbb Z_2)/\mathbb Z_2 \cong \mathbb Z_2$. We denote by $\sigma$ the non-trivial element in the latter.
Thus the fixed points in $G/K^+$ are $[e]$ and $[\sigma]$. Notice also that $\W(K^-)/\W(H) = \mathbb Z_2$, which is consistent with Lemma \ref{lem:weylquotcase2}. 
Denote the differentials of the projections $T^2 = S^1 \times S^1\to S^1$ onto the factors by $f_1$ and $f_2$. Then the roots of $\mathrm{SU}(2)^2$ are $\pm 2f_1$ and $\pm 2f_2$.

We then have that the edge on $G/K^+$ is labelled by $\pm 2f_2$. 
One of the normal edges is labelled by the differential of the character $T^2 \to T^2/\Delta S^1 = S^1$, which is given by $\pm (f_1-f_2)$,
and the Weyl group of each copy of $\mathrm{SU}(2)$ just switches the signs of either $f_1$ or $f_2$, so the remaining label is $\pm (f_1+f_2)$ (see Figure \ref{fig:graphs6}).

\begin{figure}[htb]
\begin{center}
\begin{tikzpicture}[scale=1.2]
\filldraw (0,2) circle (1.5pt) 
          (3,2) circle (1.5pt);
\draw[dotted] (0,2) .. node[yshift = 0.2cm] {$e_1$} controls (1,3) and (2,3) .. (3,2);
\draw[dotted] (0,2) .. node[yshift = 0.2cm] {$e_1+e_2$} controls (1.5,2) .. (3,2);
\draw[dotted] (0,2) .. node[yshift = 0.2cm] {$e_2$} controls (1,1) and (2,1) .. (3,2);
\end{tikzpicture} \hspace{1.5 cm}
\begin{tikzpicture}[scale=1.2]
\filldraw (0,0) circle (1.5pt) node[left] {$[\sigma]$}
          (0,2) circle (1.5pt) node[left] {$[e]$};
\draw (0,0) -- node[left] {$2f_2$} (0,2);
\draw[dotted] (0,0) .. controls (3,0) and (3,1) .. node[right] {$f_1+f_2$} (0,2);
\draw[dotted] (0,0) .. controls (3,1) and (3,2) .. node[right] {$f_1-f_2$} (0,2);
\end{tikzpicture}
\caption{On the left, the GKM graph of $S^6$ with group diagram
$(\mathrm{SU}(3),\mathrm{SU}(3),\mathrm{SU}(3),\mathrm{SU}(2))$. On the right,
the GKM graph of $S^6$ with group diagram $(\mathrm{SU}(2)^2,\mathrm{SU}(2)\times S^1,\Delta\mathrm{SU}(2),\Delta S^1)$}
\label{fig:graphs6}
\end{center}
\end{figure}
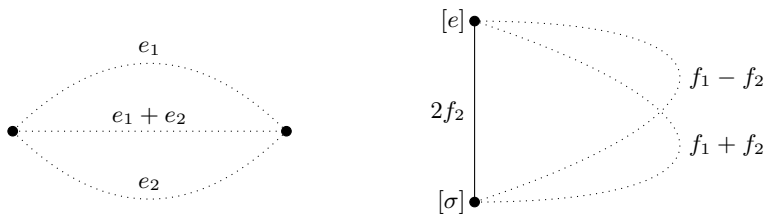
Dividing out by the kernel $\mathbb Z_2$ of the $\mathrm{SU}(2)^2$-action one gets an effective cohomogeneity one action of $\mathrm{SO}(4)$ on $S^6$ with group diagram $$(\mathrm{SO}(4), \mathrm{U}(2), \mathrm{SO}(3), S^1).$$
Observe that the $\mathrm{SO}(4)$-action extends to the transitive $G_2$-action on $S^6$, as well as the $\mathrm{SU}(3)$-action on $S^6$ considered above. As both $\mathrm{SO}(4)$ and $\mathrm{SU}(3)$ have
rank two, the induced $T^2$-actions are conjugate.
\end{example}

It was shown in \cite[Section 2]{ale_pode} and \cite[Proposition 1.23]{hoelscher2} that an effective cohomogeneity one $G$-action on a compact, simply connected manifold with exactly one fixed point is one of the following:
\begin{enumerate}
\item The ${\mathrm{SU}}(n)$- or ${\mathrm{U}}(n)$-action on ${\mathbb{C}}P^n$ 
given by $A\cdot [z_0,\ldots , z_n] = [z_0,A(z_1,\ldots,z_n)]$, with the group 
diagrams respectively given by
\[
({\mathrm{SU}}(n),{\mathrm{SU}}(n),{\mathrm{S}}({\mathrm{U}}(n-1){\mathrm{U}}
(1)),{\mathrm{SU}}(n-1)),\] and
\[({\mathrm{U}}(n),{\mathrm{U}}(n),{\mathrm{U}}(n-1){\mathrm{U}}(1),{\mathrm{U}}
(n-1)).\]
\item The ${\mathrm{Sp}}(n)$-action on ${\mathbb{H}}P^n$ given by $A\cdot 
[x_0,\ldots , x_n] = [x_0,A(x_1,\ldots,x_n)]$, with group diagram 
\[
({\mathrm{Sp}}(n),{\mathrm{Sp}}(n),{\mathrm{Sp}}(n-1){\mathrm{Sp}}(1),{\mathrm{
Sp}}(n-1)).
\]
\item The ${\mathrm{Sp}}(n)\times {\mathrm{Sp}}(1)$- or ${\mathrm{Sp}}(n)\times 
{\mathrm{U}}(1)$-action on ${\mathbb{H}} P^n$ given by $(A,p)\cdot 
[x_0,\ldots,x_n] = [px_0,A(x_1,\ldots,x_n)]$, with group diagrams respectively given by
\[
({\mathrm{Sp}}(n)\times {\mathrm{Sp}}(1),{\mathrm{Sp}}(n)\times 
{\mathrm{Sp}}(1),{\mathrm{Sp}}(n-1){\mathrm{Sp}}(1)\times 
{\mathrm{Sp}}(1),{\mathrm{Sp}}(n-1)\Delta{\mathrm{Sp}}(1))
\]
and
\[
({\mathrm{Sp}}(n)\times {\mathrm{U}}(1),{\mathrm{Sp}}(n)\times 
{\mathrm{U}}(1),{\mathrm{Sp}}(n-1){\mathrm{Sp}}(1)\times 
{\mathrm{U}}(1),{\mathrm{Sp}}(n-1)\Delta{\mathrm{U}}(1)).
\]
\item The ${\mathrm{Sp}}(n)$-action on ${\mathbb{C}}P^{2n}$ given by $A\cdot 
[z_0,x_1,\ldots,x_n] = [z_0,A(x_1,\ldots,x_n)]$, where $(z_0,x_1,\ldots,x_n) \in 
{\mathbb{C}}\oplus {\mathbb{H}}^n\cong {\mathbb{C}}^{2n+1}$, with group diagram
\[
({\mathrm{Sp}}(n),{\mathrm{Sp}}(n),{\mathrm{Sp}}(n-1){\mathrm{U}}(1), 
{\mathrm{Sp}}(n-1)).
\]
\item The ${\mathrm{Sp}}(n)\times {\mathrm{U}}(1)$-action on 
${\mathbb{C}}P^{2n}$ given by $$(A,p)\cdot [z_0,x_1,\ldots,x_n] = 
[pz_0,A(x_1,\ldots,x_n)],$$ where $(z_0,x_1,\ldots,x_n) \in {\mathbb{C}}\oplus 
{\mathbb{H}}^n\cong {\mathbb{C}}^{2n+1}$, with group diagram
\[
({\mathrm{Sp}}(n)\times {\mathrm{U}}(1),{\mathrm{Sp}}(n)\times 
{\mathrm{U}}(1),{\mathrm{Sp}}(n-1){\mathrm{U}}(1)\times {\mathrm{U}}(1), 
{\mathrm{Sp}}(n-1)\Delta{\mathrm{U}}(1)).
\]
\item The isotropy representation of the Cayley plane (see \cite{iwata})
${\mathbb{O}}P^2 = {\mathrm{F}}_4/{\mathrm{Spin}}(9)$, with group diagram
\[
({\mathrm{Spin}}(9), {\mathrm{Spin}}(9), {\mathrm{Spin}}(8), 
{\mathrm{Spin}}(7)).
\]
\end{enumerate}
In all these examples, products of groups denoted like 
${\mathrm{U}}(n-1){\mathrm{U}}(1)$ mean block-diagonally embedded groups, whereas
$\Delta$ indicates a diagonally embedded group.

\begin{proposition}\label{prop:gkmwithfp} The action of a maximal torus 
$T\subset G$ on $M$ is GKM in cases (1) (except for $G={\mathrm{SU}}(2)$), (3), (5) and (6). In the remaining cases 
it is not GKM.
\end{proposition}
\begin{proof}
The rank condition in Theorem \ref{thm:gkmconditioncohom1} is satisfied for all 
eight actions. By transitivity of the Weyl group action 
on the weights of the isotropy $G$-representation at the unique $G$-fixed point,
the condition that no element of $\Delta_G$ vanishes on 
${\mathfrak{t}}\cap {\mathfrak{h}}$ is 
equivalent to the condition that no weight is a multiple of a root of $G$. By Proposition \ref{prop:cohom1repweylgroup}, 
this is true if and only if 
this representation is not the standard representation of ${\mathrm{Sp}}(n)$ on 
${\mathbb{H}}^n$. Notice that in case (1) the group $G={\mathrm{SU}}(2)$ is isomorphic to ${\mathrm{Sp}}(1)$.
\end{proof}

\begin{remark}
The spaces acted on in these eight examples are compact rank one symmetric 
spaces. As such, all of them are homogeneous spaces of equal rank, and the 
respective maximal tori act in a GKM fashion by Theorem \ref{thm:ghzhomspace}. 
The corresponding GKM graphs were described explicitly in \cite{oli_wiemeler_1}.

The quaternionic projective space ${\mathbb{H}}P^n$ is a positive quaternion
K\"ahler manifold, with the transitive action of ${\mathrm{Sp}}(n+1)$ by quaternionic K\"ahler
automorphisms. The actions in (2) and (3) are actions of subgroups of 
${\mathrm{Sp}}(n+1)$, and as such also by quaternionic K\"ahler automorphisms. Notice that this does 
not contradict the fact that in case (3) the isotropy representations of the 
torus $T$ are not quaternionic: while quaternionic K\"ahler automorphisms just leave invariant the bundle of quaternionic structures, a quaternionic representation is required to commute with the action of the quaternions.
\end{remark}

Let us go through the list of cohomogeneity one GKM actions on simply-connected 
manifolds with one fixed point, and determine the corresponding GKM graphs. On ${\mathbb{C}}P^n$ 
there are three types of actions, those in (1) and the one in (5). One of these is given by the group diagram
\[({\mathrm{U}}(n),{\mathrm{U}}(n),{\mathrm{U}}(n-1){\mathrm{U}}(1),{\mathrm{U}}(n-1));\]
the action of the diagonal maximal torus of ${\mathrm{U}}(n)$ being the standard toric action on 
${\mathbb{C}}P^n$. Hence it is of GKM type (alternatively, use Proposition \ref{prop:gkmwithfp}). The GKM graph of this action is the one-skeleton of a 
simplex, and the labels are well-known  
(see Figure \ref{fig:graphcpn} for the GKM graph of ${\mathbb{C}}P^3$). 
\begin{figure}[htb]
\begin{center}
\begin{tikzpicture}[scale=1.45]
\draw[dotted] (2.5,1) -- node[left, yshift = 0.3cm, rotate=21.8014] {$e_1$} (0,0) -- node[left, yshift = -0.3cm, rotate=-21.8014] {$e_2$} (2.5,-1)
              (0,0) -- node[xshift=0.3cm, yshift=0.2cm] {$e_3$} (2, 0);
\draw (2.5,1) -- node[right] {$e_1-e_2$} (2.5,-1) -- node[left, xshift=0.1cm, yshift=-0.6cm, rotate=-63.43495] {$e_2-e_3$} (2, 0) -- node[left, xshift=-0.7 cm, yshift = 0.3cm, rotate=63.43495] {$e_1-e_3$} cycle;
\filldraw (0,0) circle (1.5pt) node[left, xshift=-0.1cm] {$[e]$}
          (2.5,1) circle (1.5pt) node[right, xshift=0.1cm] {$[\sigma_{13}]$}
          (2.5,-1) circle (1.5pt) node[right, xshift=0.1cm] {$[\sigma_{23}]$}
          (2,0) circle (1.5pt) node[right, xshift=0.15cm] {$[e]$};
\end{tikzpicture}\hspace{1.5 cm}
\caption{GKM graph of ${\mathbb{C}}P^3$ with group diagram $(\mathrm{U}(3),  
\mathrm{U}(3), \mathrm{U}(2)\mathrm{U}(1), \mathrm{U}(2))$
}\label{fig:graphcpn}
\end{center}
\end{figure}
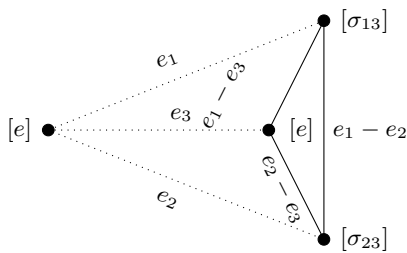

Let us check how this is compatible with our description of the graph.
The roots of ${\mathrm{U}}(n)$ are $e_i-e_j$, $1 \leq i < j \leq n$, and those that are not roots of ${\mathrm{U}}(n-1){\mathrm{U}}(1)$ are $e_i-e_n$, $1\leq i \leq n-1$. The singular orbit $G/K^+$ is a single $T$-fixed point. The singular orbit $G/K^-$ is the homogeneous space $\mathbb{C}P^{n-1}$, with the $T$-fixed points identified with the quotient of permutation groups $\W(G)/\W(K^-) = S_{n}/S_{n-1}$ of cardinality $n$. The labels of the graph of $G/K^-$ at the identity coset are exactly $\pm (e_i-e_n)$, and the remaining ones are obtained by applying the permutation group $S_n$.

The normal edge connecting $[e] \in \W(G)/\W(K^+)$ to $[e] \in \W(G)/\W(K^-)$ 
has label $\pm e_n$, as this is the differential of the character 
$T \to T/(T \cap H)$. The labels of the other normal edges are again obtained by applying $S_n$.

The other two actions on ${\mathbb{C}} P^n$, see items (1) and (5), 
are restrictions of the action above, so the graph is the same, with the labels  
restricted appropriately (as long as the action is still GKM, i.e.\ if the group is not ${\mathrm{SU}}(2)$).  

For item (3) and (6), the maximal torus of the acting group is the same as that 
of the 
homogeneous description \[
                         {\mathbb{H}}P^n = 
{\mathrm{Sp}}(n+1)/{\mathrm{Sp}}(n){\mathrm{Sp}}(1),
                        \]
and 
\[
           {\mathbb{O}}P^2 = {\mathrm{F}}_4/{\mathrm{Spin}}(9).
                        \]
So the graphs can be determined via Theorem \ref{thm:ghzhomspace}, as was done in \cite[Section 4]{oli_wiemeler_1}. See Figure~\ref{fig:graphhpn}, and for the labels, see \cite{oli_wiemeler_1}. Alternatively, one could apply our algorithm as in the case of ${\mathbb{C}}P^n$ above.

\begin{figure}[htb]
\begin{center}
\begin{tikzpicture}[scale=1.6]
\draw[dotted] (0,0) .. controls (0.9,-0.6) .. (2,-1)
               (0,0) .. controls (1.05,-0.4) .. (2,-1)
               (0,0) .. controls (0.9,0.6) ..  (2,1)
               (0,0) .. controls (1.05, 0.4) .. (2,1)
              (0,0) .. controls (0.8,0.1) .. (1.5, 0)
              (0,0) .. controls (0.8,-0.1) .. (1.5, 0);
\draw (2,1) .. controls (1.9,0) .. (2,-1)
       (2,1) .. controls (2.1,0) .. (2,-1)
        (2,-1) .. controls (1.82,-0.41) .. (1.5, 0) 
        (2,-1) .. controls (1.71,-0.62) .. (1.5, 0) 
        (2,1) .. controls (1.82, 0.41) .. (1.5, 0)
         (2,1) .. controls (1.71, 0.62) .. (1.5, 0);
\filldraw (0,0) circle (1.2pt)
          (2,1) circle (1.2pt)
          (2,-1) circle (1.2pt)
          (1.5, 0) circle (1.2pt);
\end{tikzpicture}\hspace{1.5 cm}
\begin{tikzpicture}[scale=1.6]
\draw[dotted, rotate around={-60:(2,1)}]
      (2,1) .. controls (1.93,0.5) and (1.93,-0.5) .. (2,-1)
      (2,1) .. controls (1.8,0.5) and (1.8,-0.5) .. (2,-1)
       (2,1) .. controls (2.08,0.5) and (2.08,-0.5) .. (2,-1)
       (2,1) .. controls (2.21,0.5) and (2.21,-0.5) .. (2,-1);
 \draw[dotted, rotate around={60:(2,-1)}]
      (2,1) .. controls (1.93,0.5) and (1.93,-0.5) .. (2,-1)
      (2,1) .. controls (1.8,0.5) and (1.8,-0.5) .. (2,-1)
       (2,1) .. controls (2.08,0.5) and (2.08,-0.5) .. (2,-1)
       (2,1) .. controls (2.21,0.5) and (2.21,-0.5) .. (2,-1);      
\draw (2,1) .. controls (1.93,0.5) and (1.93,-0.5) .. (2,-1)
      (2,1) .. controls (1.8,0.5) and (1.8,-0.5) .. (2,-1)
       (2,1) .. controls (2.08,0.5) and (2.08,-0.5) .. (2,-1)
       (2,1) .. controls (2.21,0.5) and (2.21,-0.5) .. (2,-1);
\filldraw (0.267949192,0) circle (1.2pt)
          (2,1) circle (1.2pt)
          (2,-1) circle (1.2pt);
\end{tikzpicture}
\caption{On the left, the GKM graph of ${\mathbb{H}P}^3$ with group diagram 
$({\mathrm{Sp}}(3),{\mathrm{Sp}}(3),{\mathrm{Sp}}(2){\mathrm{Sp}}(1),{\mathrm{
Sp}}(2))$. On the right, the GKM graph of ${\mathbb{O}}P^2$ with group diagram $({\mathrm{Spin}}(9), {\mathrm{Spin}}(9), {\mathrm{Spin}}(8), {\mathrm{Spin}}(7))$
}\label{fig:graphhpn}
\end{center}
\end{figure}
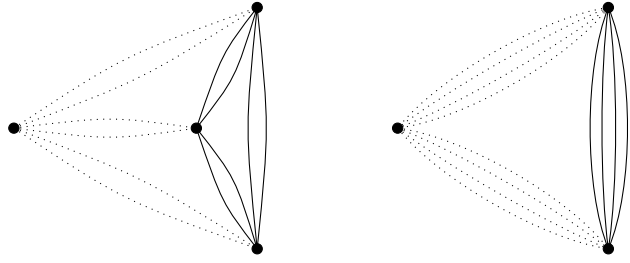

\begin{example}\label{ex:dim6}
Hoelscher classified compact simply-connected cohomogeneity-one manifolds in dimensions up to seven in \cite{hoelscher2} and \cite{hoelscherthesis}, under the assumption that the action is almost effective and nonreducible, i.e.\  the acting group does not have a proper normal subgroup that still acts with cohomogeneity one. Notice that this assumption is restrictive for our considerations: for example, in case (1) above, the maximal torus of ${\mathrm{U}}(2)$ acting on ${\mathbb{C}}P^2$ is of GKM type, but the ${\mathrm{U}}(2)$-action is reducible: the normal subgroup ${\mathrm{SU}}(2)$ still acts by cohomogeneity one, but the action of its maximal torus is not GKM. 

Let us consider this classification in dimensions up to six. In dimension two it is clear that there is only the standard circle action on $S^2$, which is of GKM type.

For dimension four consider \cite[Section 2.3]{hoelscherthesis} (notice that Parker \cite{parker} classified such actions without the assumption of simply-connectedness). The assumption on non-reducibility implies that the only possible acting groups are, up to covering, $S^3$, $S^3\times S^3$ and $S^3\times S^1$. In the case $S^3$ the maximal torus is one-dimensional and thus cannot act in a GKM fashion on a four-dimensional manifold. In the case $S^3\times S^3$ the action is an action with two fixed points on $S^4$ and falls into the category already described above, hence the action of the maximal torus is GKM. In the remaining case $S^3\times S^1$ there are two group diagrams $(S^3\times S^1, S^1\times S^1, S^1 \times S^1, S^1\times \{1\})$, and $(S^3\times S^1, S^3\times S^1, S^1 \times S^1, S^1\times \{1\})$, both of which are GKM (the second one also appears in the list of actions with fixed point above).

For dimension six, it is shown in \cite[Section 3]{hoelscher2} that the only occurring groups $G$, together with identity components $H_0$ of $H$ are those in Table \ref{Hoelscher1}.
\begin{center}

\begin{table}[htb]
\begin{tabular}{c|cc}
&$G$ & $H_0$ \\ 
\hline \\[-0.2cm]
$1_6$ & $S^3\times T^2$ & $\{1\}$\\[0.05cm]
$2_6$ & $S^3\times S^3$ & $\{(e^{ip\theta}, e^{iq\theta})\}$\\[0.05cm]
$3_6$ & $S^3\times S^3\times S^1$ & $T^2 \times \{1\}$\\[0.05cm]
$4_6$ & ${\mathrm{SU}}(3)$ & ${\mathrm{SU}}(2)$, ${\mathrm{SO}}(3)$\\[0.05cm]
$5_6$ & ${\mathrm{SU}}(3)\times S^1$ & ${\mathrm{U}}(2)\times \{1\}$\\[0.05cm]
$6_6$ & ${\mathrm{Sp}}(2)\times S^1$ & ${\mathrm{Sp}}(1){\mathrm{Sp}}(1)\times \{1\}$\\[0.05cm]
$7_6$ & ${\mathrm{Spin}}(6)$ & ${\mathrm{Spin}}(5)$ \\[0.2cm]
\end{tabular}
\caption{Pairs $(G,H_0)$ in Hoelscher's classification in dimension six}
\label{Hoelscher1}
\end{table}
\end{center}

Hoelscher also determined all possible intermediate groups $K^\pm$. For the first pair $(G,H_0)$ the rank condition in Theorem \ref{thm:gkmconditioncohom1} is not satisfied. For the second pair, the condition on the weights is not satisfied if $p= 0$ or $q= 0$. In all other cases, one checks directly that given intermediate groups $K^\pm$ satisfying the rank condition, the action of a maximal torus is GKM. Disregarding the examples that give an action on a sphere with two fixed points, we obtain the GKM group diagrams in Table \ref{Hoelscher2}, where the first column contains Hoelscher's notation for the respective type of action.
\end{example}

\begin{center}
\begin{table}[htb]
\begin{tabular}{c|ccccc}
&$G$ & $K^+$ & $K^-$ & $H$\\ \hline \\[-0.25cm]
$2_6A1$ & $S^3\times S^3$ & $S^1\times S^1$ & $S^1\times S^1$ & $\{(e^{ip\theta}, e^{iq\theta})\}\cdot {\mathbb{Z}}_n$ & $p,q\neq 0$\\[0.05cm]
$2_6C$ & $S^3 \times S^3$ & $S^1\times S^1$ & $\Delta S^3 \cdot {\mathbb{Z}}_n$ & $\Delta S^1\cdot {\mathbb{Z}}_n$ &  $n=1,2$\\[0.05cm]
$2_6D$ & $S^3\times S^3$ & $S^1\times S^1$ & $S^3\times S^1$ & $\{(e^{ip\theta},e^{i\theta})\}$ & $p\neq 0$ \\[0.05cm]
$2_6H$ & $S^3\times S^3$ & $S^3\times S^1$ & $\Delta S^3$ & $\Delta S^1$\\[0.05cm]
$2_6I$ & $S^3\times S^3$ & $S^3\times S^1$ & $S^3\times S^1$ & $\{(e^{ip\theta},e^{i\theta})\}$ & $p\neq 0$\\[0.05cm]
$2_6J$ & $S^3\times S^3$ & $S^3\times S^1$ & $S^1\times S^3$ & $\Delta S^1$\\[0.05cm]
$3_6$ & $S^3\times S^3\times S^1$ & $ S^1\times S^1\times S^1$ & $S^1\times S^1\times S^1$ & $S^1\times S^1\times \{1\}$\\[0.05cm]
$3_6$ & $S^3\times S^3\times S^1$ & $ S^1\times S^1\times S^1$ & $S^3\times S^1\times \{1\}$ & $S^1\times S^1\times \{1\}$\\[0.05cm]
$4_6$ & ${\mathrm{SU}}(3)$ & ${\mathrm{U}}(2)$ & ${\mathrm{U}}(2)$ & ${\mathrm{SU}}(2)\cdot {\mathbb{Z}}_n$\\[0.05cm]
$5_6$ & ${\mathrm{SU}}(3)\times S^1$ & ${\mathrm{U}}(2)\times S^1$ & ${\mathrm{U}}(2)\times S^1$ & ${\mathrm{U}}(2)\times \{1\}$\\[0.05cm]
$6_6$ & ${\mathrm{Sp}}(2)\times S^1$ & ${\mathrm{Sp}}(1){\mathrm{Sp}}(1)\times S^1$ & ${\mathrm{Sp}}(1){\mathrm{Sp}}(1)\times S^1$ & ${\mathrm{Sp}}(1){\mathrm{Sp}}(1)\times \{1\}$\\[0.05cm]
$6_6$ & ${\mathrm{Sp}}(2)\times S^1$ & ${\mathrm{Sp}}(2)\times S^1$ & ${\mathrm{Sp}}(1){\mathrm{Sp}}(1)\times S^1$ & ${\mathrm{Sp}}(1){\mathrm{Sp}}(1)\times \{1\}$\\[0,3cm]
\end{tabular}
\caption{Group diagrams of the GKM six-manifolds in Hoelscher's classification}
\label{Hoelscher2}
\end{table}
\end{center}

\begin{example}
Let us consider the six-dimensional example given by the group diagram
$$(G,K^+,K^-,H) = (S^3 \times S^3,
S^1\times S^1,S^3 \times S^1, \{(e^{ip\theta},e^{i\theta})\}),$$  where  $p \neq 0$, and $T=S^1\times S^1$ is the product of two copies of the standard maximal torus in $S^3$. This is the action of type $2_6D$ in Table \ref{Hoelscher2} above, and $N^6_D$ in \cite{hoelscher}.
The resulting cohomogeneity one manifold is a $\mathbb{C}P^2$-bundle over the two-sphere $S^2$, which is trivial if and only if $p\equiv 0 \,\,\, {\mathrm{mod}}\, 3$.

We have the Weyl groups $\W(K^+) = \{1\}$, $\W(K^-) = \mathbb Z_2$, $\W(G) = \mathbb Z_2 \times \mathbb Z_2$,
so we find four fixed points on the orbit $G/K^+$ and two fixed points on $G/K^-$. The four points
on $G/K^+$ give a square whose vertices are connected to the two fixed points in $G/K^-$. On the Lie algebra of $S^1$ we have the canonical linear form $\alpha$ sending $i$ to $1$. We obtain the two linear forms on the Lie algebra of $S^1\times S^1$ given by $e_1 = \alpha\oplus 0$ and $e_2 = 0\oplus \alpha$. The roots of $S^3\times S^3$ are $\pm 2e_1,\pm 2e_2$, and we denote the corresponding reflections by $\sigma_1$ and $\sigma_2$. They act on $\mathfrak t\cong {\mathbb{R}}^2$ via $\sigma_1(x,y) = (-x,y)$ and 
$\sigma_2(x,y) = (x,-y)$. Hence the vertex $[e]$ in $G/K^+$ is connected to $[\sigma_1]$
and $[\sigma_2]$ with labels $\pm 2e_1$ and $\pm 2e_2$ respectively, and these two points are both connected to $[\sigma_1\sigma_2]$
with edges labelled by $\pm \sigma_1(2e_2) = \pm 2e_2$ and $\pm \sigma_2(2e_1) = \pm 2e_1$. 
On $G/K^-$ we only have the two fixed points $[e]$ and $[\sigma_2]$ connected by an edge labelled by $\pm 2e_2$.
Notice that the edge connecting $[e] \in \W(G)/\W(K^+)$ and $[e] \in \W(G)/\W(K^-)$ is labelled by the weight $\pm (e_1-pe_2)$. The remaining labels can be found by applying the reflections $\sigma_1$
and $\sigma_2$ to $e_1-pe_2$, cf.\ Figure \ref{fig:gkm graphs}.

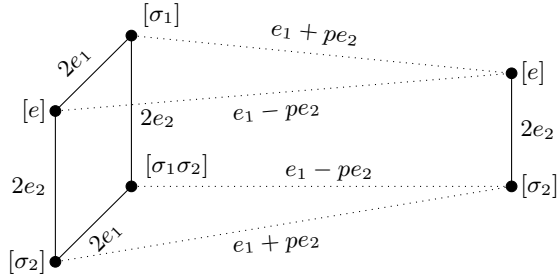
\begin{figure}[htb]
\label{6-sphere}
\begin{center}
\begin{tikzpicture}
\filldraw (0,0) circle (2pt) node[left] {$[\sigma_2]$}
          (0,2) circle (2pt) node[left] {$[e]$}
          (1,1) circle (2pt) node[above, xshift = 0.6cm] {$[\sigma_1\sigma_2]$}
          (1, 3) circle (2pt) node[above, xshift = 0.4cm] {$[\sigma_1]$}
          (6,1) circle (2pt) node[right] {$[\sigma_2]$}
          (6,2.5) circle (2pt) node[right] {$[e]$};
\draw (0,0) -- node[left] {$2e_2$} (0,2) -- node[left, yshift = 0.4cm, rotate=45] {$2e_1$} (1, 3) -- node[right, yshift = -0.1cm] {$2e_2$} (1,1) -- node[below, rotate=+45] {$2e_1$} cycle;
\draw (6,1) -- node[right] {$2e_2$} (6,2.5);
\draw[dotted] (0,0) -- node[left, xshift = 0.6cm, yshift = -0.2cm, rotate=4.76364] {$e_1+pe_2$} (6,1)
              (1,1) -- node[xshift = 0.1cm, yshift = +0.2cm] {$e_1-pe_2$} (6,1)
              (0,2) -- node[left, xshift = 0.6cm, yshift = -0.2cm, rotate=4.76364] {$e_1-pe_2$} (6,2.5)
              (1, 3) -- node[left, xshift = 0.6cm, yshift = +0.2cm, rotate=-5.7105] {$e_1+pe_2$} (6,2.5);
\end{tikzpicture}
\end{center}
\caption{GKM graph of the cohomogeneity one manifold with group diagram $(S^3 \times S^3,
S^1\times S^1,S^3 \times S^1, \{(e^{ip\theta},e^{i\theta})\})$}
\label{fig:gkm graphs}
\end{figure}
\end{example}

\begin{example}
The only three examples in Table \ref{Hoelscher2} in which $\rk{K^-}<\rk{G}$ are $2_6C$, $2_6H$, and the second $3_6$ case. The second of these was considered in Example \ref{ex:6-sphere}. Let us investigate the first, for $n=1$, i.e.\ the group diagram $(S^3\times S^3, S^1\times S^1, \Delta S^3, \Delta S^1)$. We use the notation from the previous example. Under the natural inclusion $\W(\Delta S^3)\to \W(S^3\times S^3)$ the nontrivial element is sent to $[\sigma_1\sigma_2]$, which shows, by the algorithm explained in Section \ref{subsec:secondcase}, that the normal edge emerging from $[e]$ has $[\sigma_1\sigma_2]$ as terminal vertex. We arrive at the GKM graph in Figure \ref{fig:graphdiag}.
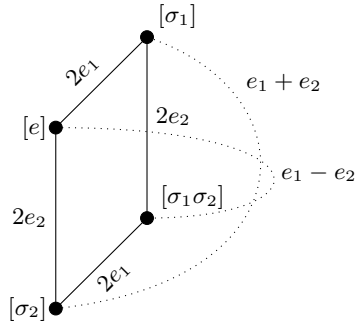
\begin{figure}[htb]
\begin{center}
\begin{tikzpicture}[scale=1.2]
\filldraw (0,0) circle (2pt) node[left] {$[\sigma_2]$}
          (0,2) circle (2pt) node[left] {$[e]$}
          (1,1) circle (2pt) node[above, xshift = 0.6cm] {$[\sigma_1\sigma_2]$}
          (1, 3) circle (2pt) node[above, xshift = 0.4cm] {$[\sigma_1]$};
\draw (0,0) -- node[left] {$2e_2$} (0,2) -- node[left, yshift = 0.4cm, rotate=45] {$2e_1$} (1, 3) -- node[right, yshift = 0.15cm] {$2e_2$} (1,1) -- node[below, rotate=+45] {$2e_1$} cycle;
\draw[dotted] (0,0) .. controls (2.8,0.3) and (2.8,2.4) .. (1,3);
\draw[dotted] (0,2) .. controls (3,2) and (3,1) .. node[right] {$e_1-e_2$} (1,1);
\node at (2.5,2.5) {$e_1+e_2$};
\end{tikzpicture}
\end{center}
\caption{GKM graph of the cohomogeneity one manifold with group diagram $(S^3\times S^3, S^1\times S^1, \Delta S^3, \Delta S^1)$}
\label{fig:graphdiag}
\end{figure}
\end{example}

\printbibliography
\end{document}